\documentclass{amsart}
\usepackage{amsmath}
\usepackage{amsfonts}
\usepackage{amssymb}
\usepackage{amsthm}
\usepackage{fullpage}
\usepackage[all]{xy}
\usepackage[hypertex]{hyperref}

\newtheorem{thm}{Theorem}[section]
\newtheorem{dfn}[thm]{Definition}
\newtheorem{lem}[thm]{Lemma}
\newtheorem{cor}[thm]{Corollary}
\newtheorem{prop}[thm]{Proposition}
\newtheorem{ass}[thm]{Assumption}

\theoremstyle{definition}
\newtheorem{rem}[thm]{Remark}

\newtheorem{construction}[thm]{Construction}

\newcommand{\qp}{\mathbb{Q}_p}
\newcommand{\cp}{\mathbb{C}_p}
\newcommand{\ocp}{\mathcal{O}_{\mathbb{C}_p}}
\newcommand{\zp}{\mathbb{Z}_p}
\newcommand{\ok}{\mathcal{O}_K}

\newcommand{\kbar}{\overline{K}}

\newcommand{\kinf}{K_{\infty}}
\newcommand{\okinf}{\mathcal{O}_{K_{\infty}}}

\newcommand{\karith}{K_{\mathrm{arith}}}

\newcommand{\kcycl}{K_{\mathrm{cycl}}}
\newcommand{\okcycl}{\mathcal{O}_{K_{\mathrm{cycl}}}}

\newcommand{\kcan}{K_{\mathrm{can}}}

\newcommand{\gk}{G_K}

\newcommand{\cc}[3]{\mathrm{H}_{\mathrm{cont}}^{#1}({#2},{#3})}
\newcommand{\coh}[3]{\mathrm{H}^{#1}({#2},{#3})}
\newcommand{\CH}[2]{\mathrm{H}(#1,#2)}
\newcommand{\CC}[2]{\mathrm{H}_{\mathrm{cont}}({#1},{#2})}
\newcommand{\ee}{\varepsilon}

\newcommand{\G}{\mathrm{Rep}_{K_{\infty}}\Gamma}

\newcommand{\IG}{\mathrm{Rep}^{\mathrm{irr}}_{K_{\infty}}\Gamma}
\newcommand{\GGG}{\mathrm{Rep}_K\Gamma}

\newcommand{\GA}{\mathrm{Rep}_{K_{\mathrm{arith}}}\Gamma^{\mathrm{arith}}}

\newcommand{\GZ}{\mathrm{Rep}^{\mathbb{Z}}_{K_{\infty}}\Gamma}
\newcommand{\GGZ}{\mathrm{Rep}^{\mathbb{Z}}_{K}\Gamma}

\newcommand{\GG}{\mathrm{Rep}_{K_{\infty}}\Gamma^{\mathrm{geom}}}

\newcommand{\Gcycl}{\mathrm{Rep}_{K_{\mathrm{cycl}}}{\Gamma^{\mathrm{cycl}}}}
\newcommand{\IGcycl}{\mathrm{Rep}^{\mathrm{irr}}_{K_{\mathrm{cycl}}}{\Gamma^{\mathrm{cycl}}}}
\newcommand{\GCK}{\mathrm{Rep}_K{\Gamma^{\mathrm{cycl}}}}

\newcommand{\GC}{\mathrm{Rep}_{\mathbb{C}_p}G_K}
\newcommand{\GCZ}{\mathrm{Rep}^{\mathbb{Z}}_{\mathbb{C}_p}G_K}
\newcommand{\GCZp}{\mathrm{Rep}^{\mathbb{Z}'}_{\mathbb{C}_p}G_K}

\newcommand{\LL}{\mathrm{Rep}\ \!\mathfrak{g}_{K_{\infty}}}
\newcommand{\LG}{\mathrm{Rep}\ \!\mathfrak{g}^{\mathrm{geom}}_{K_{\infty}}}

\newcommand{\LK}{\mathrm{Rep}\ \!\mathfrak{g}_K}

\newcommand{\IL}{\mathrm{Rep}^{\mathrm{irr}}\ \!\mathfrak{g}_{K_{\infty}}}

\newcommand{\LZ}{\mathrm{Rep}^{\mathbb{Z}}\ \!\mathfrak{g}_{K_{\infty}}}
\newcommand{\LZp}{\mathrm{Rep}^{\mathbb{Z}'}\ \!\mathfrak{g}_{K_{\infty}}}

\newcommand{\LLZ}{\mathrm{Rep}^{\mathbb{Z}}\ \!\mathfrak{g}_{K}}
\newcommand{\LLZp}{\mathrm{Rep}^{\mathbb{Z}'}\ \!\mathfrak{g}_{K}}

\newcommand{\ILA}{\mathrm{Rep}^{\mathrm{irr}}\ \!\mathfrak{g}^{\mathrm{arith}}_{K_{\infty}}}

\newcommand{\LAA}{\mathrm{Rep}\ \!\mathfrak{g}^{\mathrm{arith}}_{K_{\mathrm{arith}}}}

\newcommand{\LLCK}{\mathrm{Rep}\ \!\mathfrak{g}^{\mathrm{cycl}}_{K}}

\newcommand{\Lcycl}{\mathrm{Rep}\ \!{\mathfrak{g}^{\mathrm{cycl}}_{K_{\mathrm{cycl}}}}}

\newcommand{\RL}{\mathrm{Rep}\ \!\mathfrak{g}}

\newcommand{\R}{\Gamma}
\newcommand{\RG}{\Gamma^{\mathrm{geom}}}
\newcommand{\RA}{\Gamma^{\mathrm{arith}}}
\newcommand{\RC}{\Gamma^{\mathrm{cycl}}}

\newcommand{\Ll}{\mathfrak{g}_{K_{\infty}}}
\newcommand{\Lg}{\mathfrak{g}^{\mathrm{geom}}_{K_{\infty}}}

\newcommand{\Lc}{\mathfrak{g}^{\mathrm{cycl}}_{K_{\mathrm{cycl}}}}
\newcommand{\Lk}{\mathfrak{g}_K}

\newcommand{\LLl}{\mathfrak{g}}
\newcommand{\LLg}{\mathfrak{g}^{\mathrm{geom}}}
\newcommand{\LLa}{\mathfrak{g}^{\mathrm{arith}}}
\newcommand{\LLc}{\mathfrak{g}^{\mathrm{cycl}}}

\newcommand{\LgK}{\mathfrak{g}^{\mathrm{geom}}_K}

\newcommand{\LcK}{\mathfrak{g}^{\mathrm{cycl}}_K}

\newcommand{\ET}[2]{\mathrm{Ext}_{\mathcal{T}}^n(#1,#2)}

\newcommand{\PP}[1]{\partial (#1)}

\newcommand{\DS}[1]{D_{\mathrm{Sen}}(#1)}
\newcommand{\DSS}{D_{\mathrm{Sen}}}

\newcommand{\rr}[1]{\gamma_{#1}}
\newcommand{\xx}[1]{\mathfrak{X}_{#1}}

\newcommand{\VEC}[1]{\underline{\mathrm{Vec}}_{#1}}

\newcommand{\V}[1]{V_{[#1]}}

\newcommand{\HD}{\widehat{\Omega}_K^1}

\newcommand{\isom}[1]{{#1}/\negthickspace\sim}

\newenvironment{enuroman}{%
\begin{enumerate}%
}{\end{enumerate}}

\begin{document}

\title[A note on Sen's theory]{A note on Sen's theory in the imperfect residue field case}
\author{Shun Ohkubo}
\date{}
\address{Graduate School of Mathematical Sciences, University of Tokyo, 3-8-1 Komaba, Meguro-ku, Tokyo 153-8914, Japan}
\email{shuno@ms.u-tokyo.ac.jp}
\maketitle

\begin{abstract}
In Sen's theory in the imperfect residue field case, Brinon defined a functor from the category of $\cp$-representations to the category of linear representations of certain Lie algebra. We give a comparison theorem between the continuous Galois cohomology of $\cp$-representations and the Lie algebra cohomology of the associated representations. The key ingredients of the proof are Hyodo's calculation of Galois cohomology and the effaceability of Lie algebra cohomology for solvable Lie algebras.
\end{abstract}

\section*{Introduction}
Let $p$ be a prime and let $K$ be a CDVF of characteristic $(0,p)$ with residue field $k_K$ satisfying $[k_K:k_K^p]<\infty$. Let $\kbar$ be an algebraic closure of $K$ and let $\cp$ be the $p$-adic completion of $\kbar$. Let $t_1,\dotsc,t_d\in K$ be a lift of a $p$-basis of $k_K$ and let $\kinf=K(\mu_{p^{\infty}},t_1^{p^{-\infty}},\dotsc,t_d^{p^{-\infty}})$. Let $\gk=G_{\kbar/K}$ and $\R=G_{\kinf/K}$.

In \cite{Bri}, Brinon generalized Sen's theory to the imperfect residue field case. Let us recall it briefly. In his theory, he first established a canonical equivalence of categories
\[
\GC\to\G
\]
where $\GC$ is the category of $\cp$-representations of $\gk$ and $\G$ is the category of $\kinf$-representations of $\R$.
Then, he defined a functor
\[
\G\to\LL
\]
where $\LL$ is the category of $\kinf$-linear representations of $\LLl=\mathrm{Lie}(\R)$. Let $\DSS$ be the composite of these two functors. He also proved that, for a $\cp$-representation $V$, the canonical map
\[
\kinf\otimes_K\cc{0}{\gk}{V}\to\coh{0}{\Ll}{\DS{V}}
\]
is an isomorphism. Here, the LHS is the continuous Galois cohomology and the RHS is the Lie algebra cohomology.

The aim of this paper is to extend this isomorphism to that of $\delta$-functors:
\begin{thm}[Main Theorem]\label{thm:Main}
There exists a canonical isomorphism of $\delta$-functors
\[
\kinf\otimes_K\CC{\gk}{-}\cong\CH{\Ll}{\DS{-}}.
\]
\end{thm}

The main part of the paper is from $\S \ref{sec:Gen}$ to $\S \ref{sec:Pro}$. In $\S \ref{sec:App}$, we point out errors in \cite{Bri} and show how to fix them. We decide to include it since results of \cite{Bri} are used in the paper.

\section*{Acknowledgements}
The author would like to acknowledge the continuing guidance and encouragement of my advisor Professor Atsushi Shiho. Several helpful discussions with Professor Tsuji Takeshi and Professor Olivier Brinon are gratefully acknowledged. The author was supported by Global COE Program of University of Tokyo.
\section*{Notation and convention}
Let $p$ be a prime. Unless a particular mention is stated, let $K$ be a CDVF of mixed characteristic $(0,p)$, $k_K$ be the residue field of $K$ and assume that $[k_K:k_K^p]=p^d<\infty$. Denote by $\ok,\pi_K$ the integer ring and a uniformizer of $K$. Fix an algebraic closure $\kbar$ of $K$, let $\cp$ be a $p$-adic completion of $\kbar$ and let $G_K=G_{\kbar/K}$.

For a profinite group $G$, denote by $\mathrm{H}_{\mathrm{cont}}(G,-)$ the continuous Galois cohomology of $G$. For a Lie algebra $\mathfrak{g}$, denote by $\mathrm{H}(\mathfrak{g},-)$ the Lie algebra cohomology of $\mathfrak{g}$. For an abelian category $\mathcal{A}$, denote by $\mathrm{Ext}_{\mathcal{A}}$ the Yoneda's Ext group in $\mathcal{A}$ (put $\mathrm{Ext}_{\mathcal{A}}^0=\mathrm{Hom}_{\mathcal{A}}$). Note that if $F:\mathcal{A}\to\mathcal{A}'$ is an additive, exact functor of abelian categories, then $F$ induces a morphism of $\delta$-functors $\mathrm{Ext}_{\mathcal{A}}(A,-)\to\mathrm{Ext}_{\mathcal{A}'}(F(A),F(-))$ for $A\in\mathcal{A}$.

For a commutative topological field $R$ and a profinite groups $G$, which continuously acts on $R$, an $R$-representation of $G$ is a finite dimensional $R$-vector spaces with continuous semilinear $G$-action. Let $\mathrm{Rep}_{R}G$ be the Tannakian category of $R$-representations of $G$. $V\in\mathrm{Rep}_{R}{G}$ is admissible if $V\cong R^{\dim_{R}{V}}$ in $\mathrm{Rep}_{R}{G}$, that is, $\dim_{R^G}{V^G}=\dim_{R}{V}$. We have a canonical isomorphism $\cc{0}{G}{V}\cong\mathrm{Ext}^{0}_{\mathrm{Rep}_RG}(R,V)$ for $V\in\mathrm{Rep}_RG$ and we also have a canonical identification $\cc{1}{G}{V}\cong\mathrm{Ext}^{1}_{\mathrm{Rep}_RG}(R,V)$ which sends a $1$-cocycle $s$ to the extension
\[
\xymatrix{0\ar[r]&V\ar[r]^(.4){\iota}&V\oplus R\ar[r]^(.6){\pi}&R\ar[r]&0,}
\]
where the $G$-action on $V\oplus R$ is given by $g(v,r)=(g(v)+s(r),g(r)).$

For a field $K$, let $\mathrm{Vec}_K\ (\mathrm{resp.}\ \VEC{K})$ be the category of $K$-vector spaces (resp. finite dimensional $K$-vector spaces). For a finite dimensional Lie algebra $\LLl$ over $K$, let $\RL$ be the Tannakian category of finite dimensional linear representations of $\LLl$ over $K$.

For a Tannakian category $\mathcal{T}$, denote a unit object by $1_{\mathcal{T}}$ and the dual of $X\in\mathcal{T}$ by $X^{\vee}$. For Tannakian categories $\mathcal{T}_1,\mathcal{T}_2$, a functor $F:\mathcal{T}_1\to\mathcal{T}_2$ is called a functor of Tannakian categories if $F$ is an additive, exact, $\otimes$-functor in usual sense and commutes with ${}^\vee$.

For an abelian category $\mathcal{A}$, denote by $\mathcal{A}^{\mathrm{irr}}$ the full subcategory of irreducible objects of $\mathcal{A}$. $\isom{\mathcal{A}}$ denotes the class of the isomorphism classes of the irreducible objects of $\mathcal{A}$. In this paper, we always assume that $\isom{\mathcal{A}}$ is a set. Let $Z(\mathcal{A})$ be the free $\mathbb{Z}$-module with the basis $\isom{\mathcal{A}}$. In the case that the objects of $\mathcal{A}$ have finite length, let $l:\mathcal{A}\to Z(\mathcal{A})$ be the additive map sending $A\in\mathcal{A}$ to the formal sum of its Jordan-H\"{o}lder factors and let $\# l:\mathcal{A}\to\mathbb{Z}$ be the usual length function.
\section{General lemmas and the Lie algebra cohomology}\label{sec:Gen}
The following lemmas will be used later and the readers can skip this section for a while.
\begin{lem}\label{lem:tan}
Let $\mathcal{T}$ be a Tannakian category. Then, there exists canonical isomorphisms \[\mathrm{Ext}_{\mathcal{T}}^n(X\otimes Y,Z)\cong\mathrm{Ext}_{\mathcal{T}}^n(Y,X^{\vee}\otimes Z),\quad n\ge 0\] for $X,Y,Z\in\mathcal{T}$.
\end{lem}
\begin{proof}
We can assume $n>0$. We construct a canonical map $\phi:\ET{X\otimes Y}{Z}\to\ET{Y}{X^{\vee}\otimes Z}$ and an inverse $\psi$ as follows: For an exact sequence
\[\xymatrix{
\mathcal{E}:0\ar[r]&Z\ar[r]&W_{n-1}\ar[r]&\cdots\ar[r]&W_0\ar[r]&X\otimes Y\ar[r]&0,
}\]
define $\phi([\mathcal{E}]):=[\phi(\mathcal{E})]\in\ET{Y}{X^{\vee}\otimes Z}$ as the pull-back diagram
\[\xymatrix{
\phi(\mathcal{E}):0\ar[r]&X^{\vee}\otimes Z\ar[r]\ar@{=}[d]&W'_{n-1}\ar@{}[rd]|{\ulcorner}\ar[r]\ar[d]&\cdots\ar@{}[d]|{\cdots}\ar[r]\ar@{}[rd]|{\ulcorner}&W'_0\ar[r]\ar[d]\ar@{}[rd]|{\ulcorner}&Y\ar[r]\ar[d]&0\\
X^{\vee}\otimes\mathcal{E}:0\ar[r]&X^{\vee}\otimes Z\ar[r]&X^{\vee}\otimes W_{n-1}\ar[r]&\cdots\ar[r]&X^{\vee}\otimes W_0\ar[r]&X^{\vee}\otimes X\otimes Y\ar[r]&0
}\]
where the right vertical map is induced by the canonical map $1_{\mathcal{T}}\to X^{\vee}\otimes X.$ In the same way, for an exact sequence
\[\xymatrix{
\mathcal{E}':0\ar[r]&X^{\vee}\otimes Z\ar[r]&W'_{n-1}\ar[r]&\cdots\ar[r]&W'_0\ar[r]&Y\ar[r]&0,
}\]
define $\psi([\mathcal{E}']):=[\psi(\mathcal{E}')]\in\ET{X\otimes Y}{Z}$ as the push-out diagram
\[\xymatrix{
X\otimes\mathcal{E}':0\ar[r]&X\otimes X^{\vee}\otimes Z\ar[r]\ar[d]\ar@{}[rd]|{\lrcorner}&X\otimes W'_{n-1}\ar[r]\ar[d]\ar@{}[rd]|{\lrcorner}&\cdots\ar@{}[d]|{\cdots}\ar[r]\ar@{}[rd]|{\lrcorner}&X\otimes W'_0\ar[r]\ar[d]&X\otimes Y\ar[r]\ar@{=}[d]&0\\
\psi(\mathcal{E}'):0\ar[r]&Z\ar[r]&W_{n-1}\ar[r]&\cdots\ar[r]&W_0\ar[r]&X\otimes Y\ar[r]&0
}\]
where the left vertical map is induced by the canonical map $X\otimes X^{\vee}\to1_{\mathcal{T}}$. It is easy to check that $\phi$ and $\psi$ are well-defined and inverse to each other.
\end{proof}

\begin{lem}\label{lem:dec}
Let $\mathcal{A}$ be an abelian category whose objects have finite length. Let $Z$ be a subset of $\isom{\mathcal{A}}$ such that
\[
\mathrm{Ext}^i_{\mathcal{A}}(z,z')=0,\mathrm{Ext}^{i}_{\mathcal{A}}(z',z)=0\ \text{ for }i=0,1,\text{ for all }z\in Z,z'\in Z'
\]
with $Z'=(\isom{\mathcal{A}})\setminus Z$. Let $\mathcal{A}^{Z}$ $(\text{resp. }\mathcal{A}^{Z'})$ be the full subcategory of $\mathcal{A}$ whose objects are extensions of objects of $Z$ $(\text{resp. }Z')$. 
For $A\in\mathcal{A}$, choose a maximal subobject $A^Z\hookrightarrow A$ belonging to $\mathcal{A}^Z$ and let $p^Z(A)=A^Z$ and let $p^{Z'}(A)=A^{Z'}\in\mathcal{A}^{Z'}$ similarly.
Then, $\mathcal{A}^Z$ $(\text{resp. }\mathcal{A}^{Z'})$ is an exact, abelian subcategory of $\mathcal{A}$ and $p^{Z}:\mathcal{A}\to\mathcal{A}^{Z}$ $(\text{resp. }p^{Z'}:\mathcal{A}\to\mathcal{A}^{Z'})$ is an additive, exact functor of abelian categories and the functor
\[
(p^{Z},p^{Z'}):\mathcal{A}\to\mathcal{A}^{Z}\times\mathcal{A}^{Z'}
\]
induces an equivalence of abelian categories.
\end{lem}
\begin{proof}
Obviously, $\mathcal{A}^Z,\mathcal{A}^{Z'}$ are exact, abelian subcategories of $\mathcal{A}$. Since $A^Z\hookrightarrow A$ is a maximum subobject satisfying the condition, any given morphism $f:A_1\to A_2$ induces a unique morphism $p^Z(f):A_1^Z\to A_2^Z$ which is compatible with the injections $A_1^Z\hookrightarrow A_1,A_2^Z\hookrightarrow A_2$. Indeed, the injection $f(A_1^Z)\hookrightarrow A_2$ uniquely factors thorough $A_2^Z\hookrightarrow A_2$, the composite $A_1^Z\to f(A_1^Z)\to A_2^Z$ is the desired morphism. This implies that $p^Z$ is an additive functor and $p^{Z'}$ is also an additive functor by the same argument. The rest of the assertion follows from
\[
\mathrm{Ext}^{i}_{\mathcal{A}}(z,z')=0,\mathrm{Ext}^{i}_{\mathcal{A}}(z',z)=0\ \text{for}\ z\in\mathcal{A}^Z,z'\in\mathcal{A}^{Z'}
\]
\[
A=A^Z\oplus A^{Z'}.
\]
Let us prove these: The first part is trivial. We prove the second part by induction on $r=\# l(A)$; $r=1$ is trivial. For general $r$, first, we have $A^{Z}\cap A^{Z'}=0$. Assume $A^{Z}\neq 0$. Then, we have $(A/A^Z)^Z=0$ and a split exact sequence
\[\xymatrix{
0\ar[r]&A^Z\ar[r]&A\ar@<1mm>[r]^(.35){\pi}&(A/A^Z)^{Z'}\ar@<1mm>[l]^(.65){s}\ar[r]&0
}\]
and we have $s((A/A^Z)^{Z'})\subset A^{Z'}$, i.e., $A=A^Z+A^{Z'}$. In the case $A^Z=0$, we have $A^{Z'}\neq 0$ and the same argument works.
\end{proof}

At last, we recall about Lie algebra cohomology. Let $K$ be a field of characteristic $0$ and $\LLl$ be a finite dimensional Lie algebra over $K$. Denote the universal enveloping algebra of $\LLl$ by $U\LLl$ and regard a representation of $\LLl$ as a left $U\LLl$-module. Lie algebra cohomology is the $\delta$-functor $\mathrm{H}(\LLl,-):=\mathrm{Ext}_{U\LLl}(K,-):\RL\to\VEC{K}$. 
Note that Lie algebra cohomology commutes with a change of base fields.
\begin{dfn}[{\cite[Corollary~7.7.3]{Wei}}]\label{dfn:CE}
For $V\in\RL$, the Chevalley-Eilenberg complex $\mathrm{Hom}_{K}(\wedge^{\bullet}\LLl,V)$ is the complex
\[\xymatrix{
0\ar[r]&\mathrm{Hom}_{K}(\wedge^0{\LLl},V)\ar[r]^{d^0}&\mathrm{Hom}_{K}(\wedge^{1}{\LLl},V)\ar[r]^{d^1}&\mathrm{Hom}_{K}(\wedge^{2}{\LLl},V)\ar[r]^(.7){d^2}&\dotsb
}\]
where $d^q:\mathrm{Hom}_{K}(\wedge^q{\LLl},V)\to\mathrm{Hom}_{K}(\wedge^{q+1}{\LLl},V)$ is given by the following formula$:$

For $f\in\mathrm{Hom}_{K}(\wedge^q{\LLl},V)$,
\begin{align*}
d^q\!f(\xx{1}\wedge\dotsb\wedge\xx{q+1})=&\sum_{1\le i\le q+1}{(-1)^{i+1}\xx{i}f(\xx{1}\wedge\dotsb{\widehat{\mathfrak{X}}}_i\dotsb\wedge\xx{q+1})}\\
&+\sum_{1\le i<j\le q+1}{(-1)^{i+j}f([\xx{i},\xx{j}]\wedge\xx{1}\wedge\dotsb{\widehat{\mathfrak{X}}}_i\dotsb{\widehat{\mathfrak{X}}}_j\dotsb\wedge\xx{q+1})}.
\end{align*}
\end{dfn}

\begin{lem}[{\cite[Corollary~7.7.3]{Wei}}]\label{lem:CE}
For $V\in\RL$, we have a canonical isomorphism
\[
\coh{q}{\LLl}{V}\cong\mathrm{H}^{q}(\mathrm{Hom}_{K}{(\wedge^{\bullet}{\LLl},V))}.
\]
\end{lem}

\begin{dfn}[{\cite[Definition~7.4.3]{Wei}}]\label{dfn:der}
For $V\in\RL$, a $(K$-$)$derivation $\delta:\LLl\to V$ is a $K$-linear map such that $\delta([\mathfrak{X},\mathfrak{X}'])=\mathfrak{X}\delta({\mathfrak{X}'})-\mathfrak{X}'\delta({\mathfrak{X}})$ for $\mathfrak{X},\mathfrak{X}'\in\LLl$. Denote by $\mathrm{Der}(\LLl,V)$ the space of derivations $\LLl\to V$. Note that $\mathrm{Der}(\LLl,V)$ coincides with the space of $1$-cocycles of the Chevalley-Eilenberg complex $\mathrm{Hom}_K(\wedge^{\bullet}\LLl,V)$.
\end{dfn}



\begin{lem}[Koszul-Leger, {\cite[A special case of Lemma~3]{Hoch}}]\label{lem:KL}
If $\LLl$ is solvable, then the $\delta$-functor \[\mathrm{H}(\LLl,-):\RL\to\VEC{K}\] is effaceable.
\end{lem}

\section{$p$-adic extensions and Lie algebras}\label{sec:pad}
In this section, we fix the notation concerning $p$-adic extensions and the Lie algebras associated to them which we use later.

Let $\overline{t_1},\dotsc,\overline{t_d}\in k_K$ be a $p$-basis of $k_K$ and choose a lift $t_1,\dotsc,t_d\in\ok$ of $\overline{t_1},\dotsc,\overline{t_d}$. Let $\kinf=K(\mu_{p^{\infty}},t_1^{p^{-\infty}},\dotsc,t_d^{p^{-\infty}})$ and $H=G_{\kbar/\kinf}$ and $\R=G_{\kinf/K}$.

Let $\chi:G_K\to\zp^{\times}$ be the cyclotomic character and let $\kcycl/K$ be the cyclotomic $\zp$-extension. For $n\in\mathbb{Z}$, denote the $n$-th Tate twist by $(n)$. Let $U'=\chi(\gk),U=(1+2p\zp)\cap U'\trianglelefteq_{\mathrm{o}}\zp^{\times},T=U'\cap (\zp^{\times})_{\mathrm{tors}}.$ Let $c\in\mathbb{N}_{>0}$ such that $U=1+2p^c\zp$.

Let $\ee=(\zeta_p,\zeta_{p^2},\dotsb)$ be a generator of $\zp(1)$ and $\karith=K(\mu_{p^{\infty}})$. Let $\RG=G_{\kinf/\karith}$ and $\RA=G_{\karith/K},\RC=G_{\kcycl/K}$. Regard $\RA$ as a closed subgroup of $\R$ by the canonical injection $G_{\karith/K}\cong G_{\kinf/K(t_1^{p^{-\infty}},\dotsc,t_d^{p^{-\infty}})}\hookrightarrow\R$, which is a section of $\R\to\RA$. Then, we have an isomorphism of topological groups $\R\cong\RA\ltimes\RG$. Let $\rr{j}\in\RG$ for $1\le j\le d$ such that $\rr{j}(t_j^{p^{-n}})=\zeta_{p^n}t_j^{p^{-n}}$ for all $n\in\mathbb{N}$. We have isomorphisms $\RG\to\zp^d;\rr{j}\mapsto 1$ and $\chi:\RA\to U'$. Hence $\R\cong U'\ltimes\zp^d$, in particular, $\R$ has a $p$-adic Lie group structure. Let $\LLl=\mathrm{Lie}(\R),\LLg=\mathrm{Lie}(\RG),\LLa=\mathrm{Lie}(\RA),\LLc=\mathrm{Lie}(\RC)$. By the identification
\[
U'\ltimes\zp^d=
\begin{pmatrix}
U'&\zp&\cdots &\zp\\
 &1& &\\
 &&\ddots&\\
&&&1\\
\end{pmatrix}
\subset GL_{d+1}(\qp),
\]
we have
\[
\mathfrak{g}\cong\qp\ltimes\qp^d=
\begin{pmatrix}
\qp&\cdots&\qp \\
 && \\
 &\mbox{\smash{\huge $0$}}& 
\end{pmatrix}
\subset\mathfrak{gl}_{d+1}(\qp)
\]
where $\qp$ acts on $\qp^d$ by the scalar multiplication and $\LLg\cong\qp^d,\LLa\cong\LLc\cong\qp$.

Let $\mathrm{exp}:\mathfrak{g}\to\R$ be the exponential map whose domain is $2p^c\zp\ltimes\zp^d\subset\mathfrak{g}$ and image is $U\ltimes\zp^d\subset\R$ and let $\mathrm{log}:\R\to\mathfrak{g}$ be the log map. Let $\xx{0}=(1,0),\xx{j}=(0,e_j)\in \mathfrak{g}\cong\qp\ltimes\qp^d$ for $1\le j\le d$ with the $j$-th basic vector $e_j=(0,\dotsc,1,\dotsc,0)\in\qp^d$ and let $\rr{0}=\exp{(2p^c\xx{0})},\rr{j}=\exp{(\xx{j})}\in\R$ for $1\le j\le d$. Note that $\{\xx{i}\}_{0\le i\le d}$ is a basis of $\mathfrak{g}$ and we have $[\xx{0},\xx{j}]=\xx{j},[\xx{j},\xx{k}]=0$ for $1\le j,k\le d$.

\section{Hyodo's calculation and Sen's theory}\label{sec:Hyo}
Let $\HD:=K\otimes\varprojlim_n{\Omega^1_{\ok/\mathbb{Z}}/p^n\Omega^1_{\ok/\mathbb{Z}}}$. For $\sigma\in\gk$, let $s_j(\sigma)\in\zp$ such that $\sigma(t_j^{p^{-n}})/t_j^{p^{-n}}=\zeta_{p^n}^{s_j(\sigma)}$ for all $n\in\mathbb{N}$. Note that the continuous map $\ee^{s_j}:\gk\to\zp(1);\sigma\mapsto\ee^{s_j(\sigma)}$ is a $1$-cocycle.
\begin{thm}[{\cite[Theorem~1, Remark~3]{Hyo}}]\label{thm:hyo}
 ${ }$
\begin{enuroman}
\item For $n\in\mathbb{Z}\setminus\{0\},\cc{0}{\gk}{\cp(n)}=0$ and $\cc{0}{\gk}{\cp}\cong K$.

\item For $q>0,n\in\mathbb{Z}$, we have
\[
\dim_{K}{\cc{q}{\gk}{\cp(n)}}=
\begin{cases}
\binom{d}{n}&n=q,q-1\\
\ 0&otherwise.
\end{cases}
\]

\item $\cc{1}{\gk}{\cp}$ is generated by $\log{\chi}$ and $\cc{1}{\gk}{\cp(1)}$ is generated by $\ee^{s_j}$ for $1\le j\le d$.

\item For $q>0$, we have canonical isomorphisms
\begin{align*}
\wedge^q\cc{1}{\gk}{\cp(1)}&\stackrel{\mathrm{cup}}{\to}\cc{q}{\gk}{\cp(q)}\\
\cc{1}{\gk}{\cp}\otimes(\wedge^{q-1}\cc{1}{\gk}{\cp(1)})&\stackrel{\mathrm{cup}}{\to}\cc{q}{\gk}{\cp(q-1)}.
\end{align*}
\end{enuroman}
\end{thm}
\begin{proof}
We have only to prove the latter assertion in (iii). By \cite[Proposition~3.4.3(ii)]{Sch}, $\mathrm{d}\mathrm{log}t_j:=\mathrm{d}t_j/t_j\in\HD$ for $1\le j\le d$ is a basis of $\HD$ over $K$ and the image of $\mathrm{d}\mathrm{log}t_j$ by the canonical isomorphism $\HD\cong\cc{1}{\gk}{\cp(1)}$ is $\ee^{s_j}$ by \cite[$\S 4$]{Hyo}.
\end{proof}

\begin{thm}[{\cite[Th\'{e}or\`{e}me~1, Th\'{e}or\`{e}me~2]{Bri}}]\label{thm:sen}
${ }$
\begin{enuroman}
\item $\GC\to\mathrm{Rep}_{\widehat{\kinf}}{\R};V\mapsto V^{H}$ induces an equivalence of Tannakian categories with a quasi-inverse $V\mapsto \cp\otimes V$.

\item For $V\in\G$, let $V^{\mathrm{f}}$ be the union of $\R$-stable, finite dimensional $K$-vector subspaces of $V$. Then, $V^{\mathrm{f}}\in\G$ and $\mathrm{Rep}_{\widehat{\kinf}}{\R}\to\G;V\mapsto V^{\mathrm{f}}$ induces an equivalence of Tannakian categories with a quasi-inverse $V\mapsto \widehat{\kinf}\otimes V$.
\end{enuroman}
\end{thm}

\begin{dfn}[{\cite[$\S\S 3.2.$]{Bri}}]\label{def:dif}
For $V\in\G\ (\text{resp. }\GG,\GA,\Gcycl,\GGG)$, the differential representation $\PP{V}\in\LL\ (\text{resp. }\LG,\LAA,\Lcycl,\LK)$ of $V$ is the vector space $V$ with an action of $\mathfrak{X}\in\mathfrak{g}$ on $v\in V$ defined by
\[
\mathfrak{X}(v)=\lim_{t\to 0}{\frac{\exp{(t\mathfrak{X})}-1}{t}(v)},\quad t\in\zp,v_p(t)\gg 0.
\]
For $V\in\GC$, put $\DS{V}=\PP{(V^{H})^{\mathrm{f}}}$.
\end{dfn}

\begin{lem}[{\cite[Th\'{e}or\`{e}me~5]{Bri}}]\label{lem:sen}
\begin{enuroman}
\item The functors $\partial$ in Definition~\ref{def:dif} are well-defined and functors of Tannakian categories.

\item For $v\in V\in\G\ (\text{resp. }\GG,\GA,\Gcycl,\GGG)$, there exists an open subgroup $\Gamma_{v}\trianglelefteq_{\mathrm{o}}\R\ (\text{resp. }\Gamma^{\mathrm{geom}}_{v}\trianglelefteq_{\mathrm{o}}\RG,\Gamma^{\mathrm{arith}}_{v}\trianglelefteq_{\mathrm{o}}\RA,\Gamma^{\mathrm{cycl}}_{v}\trianglelefteq_{\mathrm{o}}\RC,\Gamma_{v}\trianglelefteq_{\mathrm{o}}\R)$ such that
\[
\gamma(v)=\exp{(\log{\gamma})(v)}
\]
for $\gamma\in\Gamma_{v}\ (\text{resp. }\Gamma^{\mathrm{geom}}_{v},\Gamma^{\mathrm{geom}}_{v},\Gamma^{\mathrm{cycl}}_{v},\Gamma_{v})$.

\item For $V\in\G\ (\text{resp. }\GG,\GA,\Gcycl,\GGG)$, $\coh{0}{\Ll}{\PP{V}}$ is\linebreak $\R\ (\text{resp. }\RG,\RA,\RC,\R)$-stable, that is, $\coh{0}{\Ll}{\PP{V}}\in\G\ (\text{resp. }\GG,\GA,\linebreak\Gcycl,\GGG)$.

\item There exists canonical isomorphisms
\begin{align*}
\kinf\otimes_{K}\mathrm{Hom}_{\G}(V_1,V_2) & \cong\mathrm{Hom}_{\LL}(\PP{V_1},\PP{V_2})\ \text{for}\ V_1,V_2\in\G\\
\kinf\otimes_{\karith}\mathrm{Hom}_{\GG}(V_1,V_2)&\cong\mathrm{Hom}_{\LG}(\PP{V_1},\PP{V_2})\ \text{for}\ V_1,V_2\in\GG\\
\karith\otimes_{K}\mathrm{Hom}_{\GA}(V_1,V_2)&\cong\mathrm{Hom}_{\LAA}(\PP{V_1},\PP{V_2})\ \text{for}\ V_1,V_2\in\GA\\
\kcycl\otimes_{K}\mathrm{Hom}_{\Gcycl}(V_1,V_2)&\cong\mathrm{Hom}_{\Lcycl}(\PP{V_1},\PP{V_2})\ \text{for}\ V_1,V_2\in\Gcycl.
\end{align*}

\item For $V_1,V_2\in\G\ (\text{resp. }\GG,\GA,\Gcycl)$,
\[
V_1\cong V_2\Leftrightarrow\PP{V_1}\cong\PP{V_2}.
\]
\end{enuroman}
\end{lem}
\begin{proof}
(i) By restricting to one-parameter subgroups (over $\zp$), the same proof of \cite[Theorem~4]{Sen} proves the convergence of the action of $\mathfrak{X}$. And the linearity of this action follows from the fact that the Lie algebras act on unit objects $(\kinf,\karith,\dotsc)$ as derivations over $K$, hence trivial, and the usual Leibniz-rule. And the action of Lie algebras actually give representations is an easy consequence from the Campbell-Hausdorff formula.

(ii) The same proof of \cite[Theorem~4]{Sen} works.

(iii) By the commutativity of $\RG,\RA,\RC$, we can assume $V\in\G,\GGG$. Then, the assertion follows from the following relations which is easy to compute (cf.\cite[Proposition~7]{Bri}):
For $\gamma\in G_{\kinf/K(t_1^{p^{-\infty}},\dotsc,t_d^{p^{-\infty}})}$,
\[
\begin{cases}
\xx{0}\ \text{commutes with}\ \gamma\\
\xx{0}\rr{j}=\rr{j}(\xx{0}+\xx{j})\ \text{for}\ 1\le j\le d,
\end{cases}
\begin{cases}
\xx{j}\gamma=\chi(\gamma)^{-1}\gamma\xx{j}\ \text{for}\ 1\le j\le d\\
\xx{j}\ \text{commutes with}\ \rr{k}\  \text{for}\ 1\le j,k\le d.
\end{cases}
\]

(iv) (\cite[Proposition~2.6(i)]{Fon1}) Let us prove the first assertion. Since $\G$ and $\LL$ are Tannakian, we can assume $V_1=K_{\infty}$ by Lemma~\ref{lem:tan}. Then, the assertion is equivalent to show the isomorphism $K_{\infty}\otimes V_2^{\R}\cong\PP{V_2}^{\Ll}$. To prove this, by replacing $V_2$ by $\PP{V_2}^{\Ll}$, which is an object of $\G$ by (iii), we can assume that $\Ll$ acts trivially on $\PP{V_2}$. By the definition of the functor $\partial$, this implies that the action of $\R$ on $V_2$ is potentially admissible. By Hilbert 90, $\R$ acts admissibly on $V_2$. The rest of the assertions are proved by the same argument.

(v) The same proof of \cite[Proposition~2.6(ii), Lemme~2.7]{Fon1} works.
\end{proof}

\section{Classification of irreducible representations}\label{sec:Cla}
\begin{lem}[Lie's theorem]\label{lem:Lie}
There exists canonical equivalences of categories
\begin{align*}
\IL & \cong \ILA \\
\IG & \cong \IGcycl.
\end{align*}
\end{lem}
\begin{proof}
Let $D\in\LL$. Since $\Ll$ is solvable and $[\Ll,\Ll]=\Lg$, $\Lg$ acts nilpotently on $D$ by Lie's theorem. If $D\in\IL$, $\Lg$ kills $D$, which implies the first assertion.

Let $V\in\G$. Then, by applying Lemma~\ref{lem:sen}(iv) with $V_1=\kinf$, we have $\kinf\otimes V^{\RG}=\PP{V}^{\Lg}$ and the RHS is non-zero. Hence, if $V\in\IG$, then $\kinf\otimes V^{\RG}=V$. So, the assertion follows from the canonical equivalence $\GA\cong\Gcycl$ (Hilbert~90).
\end{proof}

\begin{construction}[{\cite[p.41]{Fon1}}]\label{con:irr}
Denote by $\kbar/\gk$ the quotient of $\kbar$ by the canonical action of $\gk$. Let $\alpha=\{x_1,\dotsc,x_n\}\in\kbar/\gk$. Put $P_{\alpha}(X)=\Pi_{1\le i\le n}{(X-x_i)}\in K[X]$ and $K_{\alpha}=K[X]/P_{\alpha}(X)$ and let $\mathcal{O}_{K_{\alpha}}$ be its integer ring. Let $\RC_{\alpha}$ be the maximum subgroup of $\RC$ such that $v_p(\overline{X}\log{\chi(\gamma)})>1/(p-1)$ for all $\gamma\in\RC_{\alpha}$ where $\overline{X}$ is the image of $X$ in $K_{\alpha}$. Let $\chi_{\alpha}:\Gamma_{\alpha}^{\mathrm{cycl}}\to K_{\alpha}^{\times};\gamma\mapsto\exp{(\overline{X}\log{\chi({\gamma})})}$ be a character and $\rho_{\alpha}$ be the finite free $\mathcal{O}_{K_{\alpha}}$-module with the continuous linear action of $\RC$ defined by $\mathcal{O}_{K_{\alpha}}[\RC]\otimes_{\mathcal{O}_{K_{\alpha}}[\RC_{\alpha}]} \mathcal{O}_{K_{\alpha}}$ with the action on $\mathcal{O}_{K_{\alpha}}$ given by $\chi_{\alpha}$. Let $\V{\alpha}$ be an irreducible subrepresentation of $\kcycl\otimes_{\ok}\rho_{\alpha}\in\Gcycl$.
\end{construction}

\begin{lem}[{\cite[Th\'{e}or\`{e}me~2.14]{Fon1}}]\label{lem:fon}
${ }$
\begin{enuroman}
\item $\kcycl\otimes_{\ok}\rho_{\alpha}$ is isomorphic to a direct sum of copies of $\V{\alpha}$.

\item $\PP{\V{\alpha}}$ is semisimple in $\Lcycl$.

\item Regard $\isom{\Lcycl}$ as $\kbar/G_{\kcycl}$ by associating $\beta\in\kbar/G_{\kcycl}$ to $\kcycl[\xx{0}]/\Pi_{x\in\beta}(\xx{0}-x)$ and let $\pi:\kbar/G_{\kcycl}\to\kbar/\gk$ be the canonical projection. Then, we have $l(\PP{\V{\alpha}})=r\sum_{\beta\in\pi^{-1}(\alpha)}{\beta}\in Z(\Lcycl)$ with $0<r\le [\RC:\Gamma_{\alpha}^{\mathrm{cycl}}]$. 

\item The map $\kbar/\gk\to\isom{\Gcycl};\alpha\mapsto \V{\alpha}$ is bijective. In particular, $\kbar/\gk\to\isom{\G};\alpha\mapsto\kinf\otimes\V{\alpha},\kbar/\gk\to\isom{\GC};\alpha\mapsto\cp\otimes\V{\alpha}$ are bijective.
\end{enuroman}
\end{lem}
\begin{proof}
As in \cite[Proposition~2.5]{Fon1}, the characteristic polynomial of $\xx{0}$ on $V\in\Gcycl$ is $K$-coefficient. In particular, $\l(V)\in Z(\Lcycl)$ is fixed by the action of $\RC$. By a simple calculation, $\PP{\kcycl\otimes\rho_{\alpha}}$ is isomorphic to a direct sum of (at most $[\RC:\Gamma_{\alpha}^{\mathrm{cycl}}]$) copies of the semisimple object $\kcycl[\xx{0}]/P_{\alpha}(\xx{0})$. (ii) and (iii) follow from these observations. (iv) follows from (iii) and Lemma~\ref{lem:sen}(iv). Since $\PP{\kcycl\otimes\rho_{\alpha}}$ is isomorphic to a direct sum of copies of $\PP{\V{\alpha}}$, we have (i) by Lemma~\ref{lem:sen}(v).
\end{proof}

\begin{rem}\label{rem:fon}
\begin{enuroman}
\item Regard $\mathbb{Z}$ as a subset of $\kbar/\gk$. Then $\karith\otimes_{\kcycl}\V{n}\cong\karith(n)$ in $\GA$ for $n\in\mathbb{Z}$.

\item $\V{\alpha}^{\vee}\cong \V{-\alpha},\V{\alpha}\otimes \V{n}\cong \V{\alpha+n}$ for $\alpha\in\kbar/\gk,n\in\mathbb{Z}$.

\item $\PP{\V{\alpha}}$ is isomorphic to a direct sum of copies of $\kcycl[\xx{0}]/P_{\alpha}(\xx{0})$ as a $\Lc$-module.
\end{enuroman}
\end{rem}

\begin{dfn}\label{def:z}
Recall the notation of Lemma~\ref{lem:dec}. By the map $\mathbb{Z}\to\isom{\G};n\mapsto\kinf(n)$ \linebreak $(\text{resp. }\mathbb{Z}\to\isom{\GGG};n\mapsto\chi_n,\mathbb{Z}\to\isom{\LK};n\mapsto K[\xx{0}]/(\xx{0}-n))$, we regard $\mathbb{Z}$ as a subset of $\isom{\G}\ (\text{resp. }\isom{\GGG},\isom{\LK})$. Note that $\GZ\ (\text{resp. }\GGZ,\LLZ)$ is an exact, Tannakian subcategory of $\G\ (\text{resp. }\GGG,\LK)$. Also, we define $\GCZ$ in a similar way.
\end{dfn}

\section{Decomposition of categories}\label{sec:Dec}
\begin{lem}\label{lem:inj}
The canonical map \[\partial:\mathrm{Ext}^1_{\G}(V_1,V_2)\to\mathrm{Ext}^1_{\LL}(\PP{V_1},\PP{V_2})\] is injective for $V_1,V_2\in\G.$
\end{lem}
\begin{proof}
Let $\mathcal{E}$ be an exact sequence
\[
\xymatrix{0\ar[r]&V_2\ar[r]&V\ar[r]^{\pi}&V_1\ar[r]&0.}
\]
Choose basis $\{\alpha_i\}_{1\le i\le n}$ of $\mathrm{Hom}_{\G}(V_1,V)$ and $\{\beta_i\}_{1\le i\le m}$ of $\mathrm{End}_{\G}(V_1)$ with $\beta_1=\mathrm{id}_{V_1}.$ Let $P\in \mathrm{M}_{mn}(K)$ be the coefficient matrix of $\{\pi\circ\alpha_i\}_{1\le i\le n}$ with respect to $\{\beta_i\}_{1\le i\le m}.$ Then, we have 
\begin{align*}
\partial ([\mathcal{E}])=0&\Leftrightarrow\partial(\pi)\text{ has a section in } \LL\Leftrightarrow P{\bf x}={}^t(1,0,\cdots,0)\text{ for some }{\bf x}\in\kinf^n\\
&\Leftrightarrow P{\bf x}={}^t(1,0,\cdots,0)\text{ for some }{\bf x}\in K^n\Leftrightarrow\pi\text{ has a section in }\G \Leftrightarrow [\mathcal{E}]=0.
\end{align*}
\end{proof}

\begin{lem}[cf. Theorem~\ref{thm:hyo}]\label{lem:caL}
Let $\delta_0:\Lk\to K$ be the $K$-linear map defined by $(\xx{0},\xx{1},\dotsc,\xx{d})\mapsto (1,0,\dotsc,0)$ and for $1\le j\le d$, let $\delta_j:\Lk\to K[\xx{0}]/(\xx{0}-1)$ be the $K$-derivation defined by $(\xx{0},\dotsc,\xx{j},\dotsc,\xx{d})\mapsto (0,\dotsc,1,\dotsc,0)$ $(\text{cf. Definition}~\ref{dfn:der})$.
\begin{enuroman}
\item For $\alpha\in(\kbar/\gk)\setminus\{0\}$, $\coh{0}{\Lk}{K[\xx{0}]/P_{\alpha}(\xx{0})}=0$ and $\coh{0}{\Lk}{K}\cong K$.

\item For $q>0,\alpha\in\kbar/\gk$, we have
\[
\dim_{K}{\coh{q}{\Lk}{K[\xx{0}]/P_{\alpha}(\xx{0})}}=
\begin{cases}
\binom{d}{\alpha}&\alpha=q,q-1\\
\ 0&otherwise .
\end{cases}
\]

\item $\coh{1}{\Lk}{K}$ is generated by $\delta_0$ and $\coh{1}{\Lk}{K[\xx{0}]/(\xx{0}-1)}$ is generated by $\delta_1,\dotsc,\delta_d$.

\item For $q>0$, we have canonical isomorphisms
\begin{align*}
\wedge^q\coh{1}{\Lk}{K[\xx{0}]/(\xx{0}-1)}&\stackrel{\mathrm{cup}}{\to}\coh{q}{\Lk}{K[\xx{0}]/(\xx{0}-q)}\\
\coh{1}{\Lk}{K}\otimes(\wedge^{q-1}\coh{1}{\Lk}{K[\xx{0}]/(\xx{0}-1)})&\stackrel{\mathrm{cup}}{\to}\coh{q}{\Lk}{K[\xx{0}]/(\xx{0}-(q-1))}.
\end{align*}
\end{enuroman}
\end{lem}
\begin{proof}
Recall the notation of Definition~\ref{dfn:CE}. Let $q\ge 0$ and let $D\in\LK$ which is killed by $\LgK$. Then, for $\delta\in\mathrm{Hom}_{K}(\wedge^q\Lk,D)$, we have
\begin{gather*}
d^q\!\delta(\xx{0}\wedge\xx{i_1}\wedge\dotsb\wedge\xx{i_q})=(\xx{0}-q)\delta(\xx{i_1}\wedge\dotsb\wedge\xx{i_q})\\
d^q\!\delta(\xx{i_1}\wedge\dotsb\wedge\xx{i_{q+1}})=0
\end{gather*}
for $1\le i_1,\dotsc,i_{q+1}\le d$. Hence, together with Lemma~\ref{lem:CE}, (i) follows immediately and we have \linebreak$\coh{q}{\Lk}{K[\xx{0}]/P_{\alpha}(\xx{0})}=0$ for $\alpha\neq q,q-1$ with $q>0$. Moreover, for $q>0$, $\{\delta_{i_1}\cup\dotsb\cup\delta_{i_q}\}_{1\le i_1<\dotsb <i_q\le d}$ is a basis of $\coh{q}{\Lk}{K[\xx{0}]/(\xx{0}-q)}$ and $\{\delta_0\cup\delta_{i_1}\cup\dotsb\cup\delta_{i_{q-1}}\}_{1\le i_1<\dotsb <i_{q-1}\le d}$ is a basis of $\coh{q}{\Lk}{K[\xx{0}]/(\xx{0}-(q-1))}$, which implies the rest of the assertion.
\end{proof}

\begin{cor}\label{cor:calirred}
Let $\mathbb{Z}'=(\kbar/\gk)\setminus\mathbb{Z}$ and let $(\alpha,\beta)\in\mathbb{Z}\times\mathbb{Z}'$ or $\mathbb{Z}'\times\mathbb{Z}$. Then
\begin{gather*}
\mathrm{Ext}^{i}_{\G}(\kinf\otimes\V{\alpha},\kinf\otimes\V{\beta})=0\\
\mathrm{Ext}^{i}_{\LL}(\kinf[\xx{0}]/P_{\alpha}(\xx{0}),\kinf[\xx{0}]/P_{\beta}(\xx{0}))=0
\end{gather*}
for $i=0,1$.
\end{cor}
\begin{proof}
Since $(K[\xx{0}]/P_{\alpha}(\xx{0}))^{\vee}\otimes{K[\xx{0}]/P_{\beta}(\xx{0})}\cong K[\xx{0}]/P_{\beta-\alpha}(\xx{0})$, the assertion follows from Lemma~\ref{lem:tan}, Remark~\ref{rem:fon}(iii), Lemma~\ref{lem:inj} and Lemma~\ref{lem:caL}(i),(ii).
\end{proof}

Note that $\kbar/\gk\to\isom{\LK};\alpha\mapsto K[\xx{0}]/P_{\alpha}(\xx{0})$ is a bijection by Lie's theorem as in Lemma~\ref{lem:Lie}. The following corollary is an easy consequence of Lemma~\ref{lem:dec} and Lemma~\ref{lem:fon}(iv), Corollary~\ref{cor:calirred}.

\begin{cor}\label{cor:dec}
Let $p^{\mathbb{Z}},p^{\mathbb{Z}'}$ be as in Lemma~\ref{lem:dec}. We have canonical equivalences of abelian categories
\begin{align*}
(p^{\mathbb{Z}},p^{\mathbb{Z}'}):\GC&\cong\GCZ\times\GCZp\\
(p^{\mathbb{Z}},p^{\mathbb{Z}'}):\LK&\cong\LLZ\times\LLZp.
\end{align*}
\end{cor}

Our aim of the rest of this section is to prove that $\GZ$ is equivalent to $\LLZ$ as a Tannakian category.
\begin{prop}\label{prop:H1}
For $V\in\GGZ$, the canonical map \[\cc{1}{\R}{V}\to\cc{1}{\R}{V_{\kinf}}\] is an isomorphism.
\end{prop}

We prepare for the proof of this proposition:
\begin{construction}
We define a functor of Tannakian categories $\mathbb{V}:\LLZ\to\GGZ$ as follows. For $D\in\LLZ$, $\mathbb{V}(D)$ be the vector space $D$ with an action of $\R$ given by, for $v\in D$, $\gamma(v)=\exp{(\log{\gamma})}(v)$. We will see that $\mathbb{V}$ is well-defined: We have only to prove that the RHS converges.

Let $\xx{0}=\xx{0}^{\mathrm{ss}}+\xx{0}^{\mathrm{nil}}$ be the Jordan decomposition of $\xx{0}$ on $D$. Take an $\xx{0}$-stable flag of $D$. By taking a basis associated to this flag, regard $\mathfrak{gl}(D)$ as $\mathrm{M}_{r}(K)$ with $r=\dim_{K}{D}$ as a $K$-vector space and $v$ be the valuation $v(X)=\inf_{1\le i,j\le r}{v_p(x_{i,j})}$ for $X=(x_{i,j})_{1\le i,j\le r}\in \mathrm{M}_r(K)$. Note that $v(\xx{0}^{\mathrm{ss}})\ge 0$.

Let $\log{\R}:=\mathrm{Im}(\log:\R\to\LLl)\cong2p^c\zp\ltimes\zp^d\subset\qp\ltimes\qp^d$. To check the convergence, we will see that, for $\mathfrak{X}\in\log{\R}$, $\lim{v((n!)^{-1}\mathfrak{X}^n)}=\infty$. Let $\mathfrak{X}=i_0\xx{0}+i_1\xx{1}+\dotsb+i_d\xx{d}$ with $i_0\in2p^c\zp,i_1\dotsc,i_d\in\zp$. Take $M\gg 0$ such that $(\xx{0}^{\mathrm{nil}})^M=0,(\xx{j})^M=0$ for $1\le j\le d$. We can see that, for $n\gg 0$, a product of length $n$ consisting of $i_0\xx{0},\dotsc,i_d\xx{d}$ is written in the form, for some $k>n-dM$,
\begin{multline*}
(\text{a product of length }n-k\text{  consisting of } i_1\xx{1},\dotsc,i_d\xx{d})\\ \times (\text{a product of length }k \text{  consisting of the elements of a form }i_0\xx{0}+i_0a,\ a\in\mathbb{Z}).
\end{multline*} 
This follows from the relation $\xx{0}\xx{j}=\xx{j}(\xx{0}+1)$ for $1\le j\le d$ and the commutativity of $\xx{1},\dotsc,\xx{d}$. And, for $k\in\mathbb{N}$, $\text{a product of length }k\text{ consists of }i_0\xx{0}+i_0a,a\in\mathbb{Z}$ is written as a linear combination over $\mathbb{Z}$ of the elements $\{i_0^k(\xx{0}^{\mathrm{nil}})^{m}(\xx{0}^{\mathrm{ss}})^{k-m}|0\le m<M\}$. Hence, we have
\[
v\left(\frac{1}{n!}\mathfrak{X}^n\right)=v\left(\frac{1}{n!}(i_0\xx{0}+\dotsb+i_d\xx{d})^n\right)\ge C+(n-dM)v_p(2p^c)-v_p(n!)
\]
with the constant $C:=\inf_{0\le j_0,\dotsc,j_d< M}{v(\xx{1}^{j_1}\dotsb\xx{d}^{j_d}(\xx{0}^{\mathrm{nil}})^{j_0})}$, which implies the assertion.
\end{construction}

\begin{prop}\label{prop:equ}
$\partial:\GGZ\to\LLZ$ gives an equivalence of Tannakian categories with a quasi-inverse $\mathbb{V}$.
\end{prop}
\begin{proof}
Let $\Gamma^0\trianglelefteq_{\mathrm{o}}\R$ be the image of the exponential map $\mathfrak{g}\to\R$ and define $\mathrm{Rep}_{K}^{\mathbb{Z}}{\Gamma^0}$ similar to Definition~\ref{def:z}. 
Since $\partial\circ\mathbb{V}=\mathrm{id}_{\LLZ}$, we have only to prove that $\mathbb{V}$ is an equivalence.

It is easy to see that the restriction functor $|_{\Gamma^0}:\GGZ\to\mathrm{Rep}_{K}^{\mathbb{Z}}{\Gamma^0}$ is faithful by Lemma~\ref{lem:tan}. Together with the fact that the composite functor
\[
\mathbb{V}|_{\Gamma^0}:\LLZ\to\GGZ\to\mathrm{Rep}_{K}^{\mathbb{Z}}{\Gamma^0}
\]
is an equivalence, $\mathbb{V}$ is fully faithful.

Since $\mathrm{Im}\mathbb{V}$ contains all the  irreducible representations of $\GGZ$, the assertion that $\mathbb{V}$ is essentially surjective reduces to, by d\'{e}vissage, proving that $\mathrm{Ext}^1_{\LLZ}(D_1,D_2)\to\mathrm{Ext}^1_{\GGZ}(\mathbb{V}(D_1),\mathbb{V}(D_2))$ is surjective for $D_1,D_2\in\LLZ$. To prove the latter assertion, we have only to prove the injectivity of $\mathrm{Ext}^1_{\GGZ}(V_1,V_2)\to\mathrm{Ext}^1_{\mathrm{Rep}_{K}^{\mathbb{Z}}{\Gamma^0}}(V_1|_{\Gamma^0},V_2|_{\Gamma^0})$ for $V_1,V_2\in\mathrm{Im}\mathbb{V}$. By Lemma~\ref{lem:tan}, we can assume $V_1=K$. By d\'{e}vissage and the inflation-restriction sequence, the assertion follows from $\chi_n^{\R^0}=0$ for $n\in\mathbb{Z}\setminus\{0\}$ and $\cc{1}{\R/\R^0}{K}\cong\mathrm{Hom}(\R/\R^0,K)=0$.
\end{proof}

\begin{dfn}[{\cite[$\S\S 2.3$]{Fon1}}]
Let $K[\log{t}]$ be a polynomial ring with a $\RC$-action given by
\[
\gamma(\log{t})=\log{t}+\log{\chi(\gamma)}.
\]
$(i.e.,\ t=\log{\ee}.)$ Note that $K[\log{t}]^{<r}:=\oplus_{0\le n<r}{K(\log{t})^n}\in\GCK$ and $\PP{K[\log{t}]^{<r}}\cong K[\xx{0}]/\xx{0}^r$ as a $\LcK$-module.
\end{dfn}

\begin{prop}\label{prop:clas}
If $V\in\GGZ$ satisfies $l(V)\in\mathbb{Z}\cdot 0$, then $V\cong K[\log{t}]^{<r_1}\oplus\dotsb\oplus K[\log{t}]^{<r_k}$ in $\GGZ$.
\end{prop}
\begin{proof}
Note that $K[\log{t}]^{<r}\in\GGZ$. Then, by Proposition~\ref{prop:equ}, we have only to prove the following: If $D\in\LLZ$ satisfies $l(D)\in\mathbb{Z}\cdot 0$, then $D\cong K[\xx{0}]/\xx{0}^{r_1}\oplus\dotsb\oplus K[\xx{0}]/\xx{0}^{r_k}$. Since in the case $\dim_{K}{D}=1$ we have $D\cong K$, by d\'{e}vissage, we have only to prove that, for $D=K[\xx{0}]/\xx{0}^{r_1}\oplus\dotsb\oplus K[\xx{0}]/\xx{0}^{r_k}$, the canonical map $\mathrm{Ext}^{1}_{\LLCK}(K,D)\to\mathrm{Ext}^1_{\LK}(K,D)$ is an isomorphism. This follows from the same calculation in the proof of Lemma~\ref{lem:caL}.
\end{proof}

Combining the following two lemmas, Proposition~\ref{prop:H1} follows.
\begin{lem}\label{lem:extone}
For $V\in\GZ$, the canonical map
\[
\kinf\otimes\partial:\kinf\otimes\mathrm{Ext}^1_{\G}(\kinf,V)\to\mathrm{Ext}^1_{\LL}(\kinf,\PP{V})
\]
is injective.
\end{lem}
\begin{proof}
By d\'{e}vissage, we can assume $V=\kinf(n)$ for $n\in\mathbb{Z}$. In this case, we will prove that the above map is an isomorphism. By Theorem~\ref{thm:sen}, the canonical map
\[
\cc{1}{\R}{\kinf(n)}\to\cc{1}{\gk}{\cp(n)}.
\]
is an isomorphism. Then, by Theorem~\ref{thm:hyo}, we have $\cc{1}{\R}{\kinf(n)}=0$ for $n\neq 0,1$.

$\bullet\ n=0$ : By Theorem~\ref{thm:hyo}(iii), $\cc{1}{\R}{\kinf}$ is a $1$-dimensional $K$-vector space generated by $\log{\chi}$. The corresponding extension is
\[
\xymatrix{
\log{\chi}:0\ar[r]&\kinf\ar[r]^(.4){\iota_1}&\kinf\oplus\kinf\ar[r]^(.6){\pi_2}&\kinf\ar[r]&0,\\
}
\]
where the $\R$-action on $\kinf\oplus\kinf$ is given by
\[
\gamma(1,0)=(1,0),\ \gamma(0,1)=(\log{\chi(\gamma)},1).
\]
Hence the matrix of the actions of $\xx{0},\xx{1},\dotsc,\xx{d}$ on the basis $\{(1,0),(0,1)\}$ of $\PP{\kinf\oplus\kinf}$ are given by
\[
\begin{pmatrix}
0&1\\
0&0
\end{pmatrix}_,
\begin{pmatrix}
0&0\\
0&0
\end{pmatrix}_{,\dotsc,}
\begin{pmatrix}
0&0\\
0&0
\end{pmatrix}_.
\]

Then, the canonical map $\kinf\otimes_{K}{\cc{1}{\R}{\kinf}}\to\coh{1}{\Ll}{\kinf}$ is an isomorphism by Lemma~\ref{lem:caL} since this map sends $[\log{\chi}]$ to $[\delta_0]$.

$\bullet\ n=1$ : By Theorem~\ref{thm:hyo}(iii), $\cc{1}{\R}{\kinf(1)}$ is a $d$-dimensional $K$-vector space generated by $\ee^{s_j}$ for $1\le j\le d$. The corresponding extension of $\ee^{s_j}$ is
\[\xymatrix{
\ee^{s_j}:0\ar[r]&\kinf(1)\ar[r]^(.4){\iota}&\kinf(1)\oplus\kinf\ar[r]^(.65){\pi}&\kinf\ar[r]&0,
}\]
where the $\R$-action on $\kinf(1)\oplus\kinf$ is given by
\[
\gamma(1,0)=(\log{\chi(\gamma)},0),\ \gamma(0,1)=(s_j(\gamma),1).
\]
Hence the matrix of the actions of $\xx{0},\xx{1},\dotsc,\xx{j},\dotsc,\xx{d}$ on the basis $\{(1,0),(0,1)\}$ of $\PP{\kinf(1)\oplus\kinf}$ are given by
\[
\begin{pmatrix}
1&0\\
0&0
\end{pmatrix}_,
\begin{pmatrix}
0&0\\
0&0
\end{pmatrix}_{,\dotsc,}
\begin{pmatrix}
0&1\\
0&0
\end{pmatrix}_{,\dotsc,}
\begin{pmatrix}
0&0\\
0&0
\end{pmatrix}_.
\]
Hence, the canonical map $\kinf\otimes_{K}\cc{1}{\R}{\kinf(1)}\to\coh{1}{\Ll}{\kinf[\xx{0}]/(\xx{0}-1)}$ is an isomorphism by Lemma~\ref{lem:caL} since this map sends $[\ee^{s_j}]$ to $[\delta_j]$ for $1\le j\le d$.
\end{proof}

\begin{lem}
For $V\in\GGZ$, there exists the following commutative diagram
\[\xymatrix{
\kinf\otimes\mathrm{Ext}^1_{\GGZ}(K,V)\ar[r]^(.47){\kinf\otimes\partial}_(.47){\cong}\ar[d]^{\mathrm{can.}}&\kinf\otimes\mathrm{Ext}^1_{\LLZ}(K,\PP{V})\ar[d]^{\mathrm{can.}}_{\cong}\\
\kinf\otimes\mathrm{Ext}^1_{\G}(\kinf,V_{\kinf})\ar[r]^(.51){\kinf\otimes\partial}&\mathrm{Ext}^1_{\LL}(\kinf,\PP{V}_{\kinf}).
}\]
\end{lem}

\begin{construction}
Let $V\in\GZ$. We will prove that there exists a maximum element $V_0$ in the set
\[
\{W\subset V|W \text{ is a finite dimensional }\R\text{-stable }K\text{-vector space such that }l(W)\in\oplus_{n\in\mathbb{Z}}{\mathbb{Z}\cdot n\subset Z(\GGG)}\}
\]
and the canonical map $\kinf\otimes V_0\to V$ is an isomorphism. We proceed by induction on $r=\dim_{\kinf}{V}$. When $r=1$, we can assume $V=\kinf$ by tensoring with $\V{n}$ for some $n\in\mathbb{Z}$. We have only to prove $V_0=K$. Let $W\subset\kinf$ be a non-zero $K$-vector space satisfying the condition above. 
Then, the action of $\R$ on $W$ factors through a finite quotient. Since the action of $\R$ on $\chi_n$ for $n\in\mathbb{Z}$ factors through a finite quotient if and only if $n=0$, we have $l(W)\in\mathbb{Z}\cdot 0$. By Proposition~\ref{prop:clas}, $W$ is a direct sum of representations of the form $K[\log{t}]^{<r}$. Since the action of $\R$ on $K[\log{t}]^{<r}$ factors through a finite quotient if and only if $r=1$, $\R$ acts on $W$ trivially, i.e., $W=K$. For general $r$, after tensoring with some $\V{n},\ n\in\mathbb{Z}$, we have an exact sequence in $\G$
\[\xymatrix{
\mathcal{E}:0\ar[r]&V'\ar[r]^{\alpha}&V\ar[r]^{\beta}&\kinf\ar[r]&0.
}\]
By Proposition~\ref{prop:H1}, there exists a unique extension
\[\xymatrix{
\mathcal{E}_0:0\ar[r]&V'_0\ar[r]&V_0\ar[r]&K\ar[r]&0
}\]
of $K$ by $V'_0$ in $\GGG$ such that $[\mathcal{E}]=\kinf\otimes[\mathcal{E}_0]$. If $V_0\subset W$ satisfies the condition above, then we have $\alpha^{-1}(W)\subset V'_0,\beta(W)\subset K$, i.e., $W=V_0$.
\end{construction}

Thus, we have
\begin{prop}
The functor $\GZ\to\GGZ;V\mapsto V_0$ induces an equivalence of Tannakian categories with a quasi-inverse $V\mapsto\kinf\otimes V$.
\end{prop}

\section{Calculation of Galois cohomology}\label{sec:Cal}
Our aim of this section is to prove the following theorem:
\begin{thm}\label{thm:van}
For $V\in\GCZp$, $\CC{\gk}{V}=0$.
\end{thm}

For a finite extension $K'/K$, regard $\overline{K'}=\overline{K}$ and $G_{K'}\subset\gk$. Then, for $V\in\GCZp$, we have $V|_{G_{K'}}\in\mathrm{Rep}_{\cp}^{\mathbb{Z}'}{G_{K'}}$: To prove this, we can assume that $K'/K$ is Galois and $V$ is irreducible, i.e., $V=\cp\otimes\V{\alpha}$ for $\alpha\in(\kbar/\gk)\setminus\mathbb{Z}$. And we have only to prove that $\mathrm{Hom}_{\mathrm{Rep}_{\cp}{G_{K'}}}(\cp(n),\cp\otimes\V{\alpha})=0$ for all $n\in\mathbb{Z}$. By Lemma~\ref{lem:tan} and Remark~\ref{rem:fon}(ii), we can assume $n=0$. In this case, the assertion follows from
\[
\dim_{K'}{\cc{0}{G_{K'}}{\cp\otimes\V{\alpha}}}=\dim_{K}{\cc{0}{\gk}{\cp\otimes\V{\alpha}}}=0
\]
where the first equality follows from Hilbert 90 and the latter one follows from Corollary~\ref{cor:calirred}. Hence, to prove the theorem, we can replace $K$ by its finite extension satisfying Assumption~\ref{ass}. (Such an extension exists by a theorem of \cite{Epp}.) In particular, we have $\karith=\kcycl$ and we can apply the argument of \cite[\S 2]{Bri} (cf. Appendix). Moreover, by the canonical isomorphism $\CC{\gk}{V}\cong\CC{\R}{V^H}$ (see \cite[Lemma~(3-5)]{Hyo}) and d\'{e}vissage, we have only to prove
\begin{prop}\label{prop:vani}
For $\alpha\in(\kbar/\gk)\setminus\mathbb{Z}$, we have $\CC{\R}{\widehat{\kinf}\otimes\V{\alpha}}=0$.
\end{prop}

For $1\le j\le d$, let $\R^{(j)}:=\overline{\langle\rr{j}\rangle}\trianglelefteq\RG$ and $\kinf^{(j)}:=K(\mu_{p^{\infty}},t_1^{p^{-\infty}},\dotsc,t_j^{p^{-\infty}})$. For $1\le j\le d$, let $\theta^{(j)}:=\varinjlim_{n}{{[\kinf^{(j-1)}(t_j^{p^{-n}}):\kinf^{(j-1)}]}^{-1}\mathrm{Tr}_{\kinf^{(j-1)}(t_j^{p^{-n}})/\kinf^{(j-1)}}}:\kinf^{(j)}\to\kinf^{(j-1)}$ be the normalized trace map of Tate.
\begin{lem}\label{lem:Tr}
\begin{enuroman}
\item $\theta^{(j)}$ is a bounded $\kinf^{(j-1)}$-linear map. Passing to the completion, denote again $\theta^{(j)}:\widehat{\kinf^{(j)}}\to\widehat{\kinf^{(j-1)}}$. Let $X_j:=\ker{\theta^{(j)}},X^{\mathrm{int}}_j:=X_j\cap\ocp$.

\item $\CC{\R^{(j)}}{X_j/X^{\mathrm{int}}_j}$ and $\CC{\RG}{X_j/X^{\mathrm{int}}_j}$ are killed by some power of $p$.
\end{enuroman}
\end{lem}
\begin{proof}
(i) In \cite[Before Lemme~4]{Bri}, it is proved that the normalized trace map $\kinf\to \linebreak K(\mu_{p^{\infty}},t_1^{p^{-\infty}},\dotsc,\widehat{t_j^{p^{-\infty}}},\dotsc,t_d^{p^{-\infty}})$ (defined similarly as above) is a bounded $K(\mu_{p^{\infty}},t_1^{p^{-\infty}},\dotsc,\widehat{t_j^{p^{-\infty}}},\dotsc,t_d^{p^{-\infty}})$-linear map. By restricting this map to $K^{(j)}_{\infty}$, we obtain the assertion.

(ii) To prove the first assertion, we can assume that the degree is $0$ or $1$. Then, the assertion follows from the same argument of the proofs of \cite[Proposition~7,8]{Tate}.
The latter assertion follows from the first one by using the Hochschild-Serre spectral sequence.
\end{proof}

\begin{cor}[cf. {\cite[Proposition~(2-1)]{Hyo}}]\label{cor:hyo}
\begin{enuroman}
\item The kernel and cokernel of the canonical map
\[
\cc{q}{\RG}{\kcycl/\okcycl}\to\cc{q}{\RG}{\kinf/\okinf}
\]
are killed by some power of $p$.

\item Let $\rho$ be a finite free $\ok$-module with a continuous linear action of $\RC$. Then, we have a canonical $\RC$-equivariant isomorphism
\[
\cc{q}{\RG}{\kcycl/\okcycl\otimes\rho}\cong\wedge^q(\mathbb{Z}^d)\otimes\kcycl/\okcycl\otimes\rho(-q).
\]
\end{enuroman}
\end{cor}
\begin{proof}
(i) follows from the decomposition of discrete $\RG$-modules $\kinf/\okinf=\kcycl/\okcycl\oplus X_1/X_1^{\mathrm{int}}\oplus\dotsb\oplus X_d/X_d^{\mathrm{int}}$ and (ii) follows from a direct computation.
\end{proof}

Let $\theta^{(0)}:=\varinjlim_{n}{{[K(\mu_{p^n}):K]}^{-1}\mathrm{Tr}_{K(\mu_{p^n})/K}}:\kcycl\to K$ and $\mathbb{K}$ be the $p$-adic completion of $K(t_1^{p^{-\infty}},\dotsc,t_d^{p^{-\infty}})$. Note that $\mathbb{K}$ is a CDVF with a perfect residue field.
\begin{lem}[{\cite[Remark of (3.1), Proposition~7]{Tate}}]\label{lem:tr}
${ }$
\begin{enuroman}
\item $\theta^{(0)}$ is a bounded $K$-linear map.

\item Let $A\in GL_r(\ok)\subset GL_r(\widehat{\kcycl})$ such that $A\equiv 1\mod{\pi_K}$ and $1-A^{p^n}$ is invertible for all $n\in\mathbb{N}$. Then, for $\gamma\in\RC\setminus\{1\}$, $\gamma-A$ induces an automorphism of a $p$-adic Banach space on $\widehat{\kcycl}^r$.
\end{enuroman}
\end{lem}
\begin{proof}
(i) Since the normalized trace map $\theta^{(0)}:\mathbb{K}_{\mathrm{cycl}}\to\mathbb{K}$ is a bounded $\mathbb{K}$-linear map by \cite[Remark of (3.1)]{Tate}, we obtain the assertion by restricting $\theta^{(0)}$ to $\kcycl$.

(ii) The same argument of \cite[Proposition~7]{Tate} works.
\end{proof}

\begin{lem}[{\cite[Lemma~(3-6)]{Hyo}}]
Let $G$ be a profinite group and $V$ be a topological $G$-module with a $G$-submodule $V'$ such that $V=\varprojlim_{n}{V/p^nV'}$. Then, for $q\ge 0$, we have canonical exact sequence
\[
\xymatrix{0\ar[r]&\varprojlim^1_n{{\cc{q-1}{G}{V/p^nV'}}}\ar[r]&\cc{q}{G}{V}\ar[r]&\varprojlim_{n}\cc{q}{G}{V/p^nV'}\ar[r]&0.}
\]
\end{lem}
\begin{proof}[Proof of Proposition~\ref{prop:vani}]
Recall Construction~\ref{con:irr}. Applying the above lemma to $V=\kinf\otimes_{\ok}\rho_{\alpha},V'=\mathcal{O}_{\kinf}\otimes_{\ok}\rho_{\alpha}$, we have only to prove that $\cc{q}{\R}{\kinf/\mathcal{O}_{\kinf}\otimes_{\ok}\rho_{\alpha}}$ is killed by some power of $p$. By the Hochschild-Serre spectral sequence and Corollary~\ref{cor:hyo} and Remark~\ref{rem:fon}(ii), we have only to prove that $\cc{q}{\RC}{\kcycl/\okcycl\otimes\rho_{\alpha}}$ is killed by some power of $p$ for $q=0,1$. To prove this, we can replace $\RC$ by its open subgroup $\Gamma^{\mathrm{cycl}}_{\alpha}$.

Let $r=\#\alpha$, $m=[\RC:\Gamma^{\mathrm{cycl}}_{\alpha}]$. For a topological generator $\gamma_{\alpha}\in\Gamma^{\mathrm{cycl}}_{\alpha}$, let $A\in GL_{r}(\ok)$ be the matrix of $\gamma_{\alpha}^{-1}$ on $\rho_{\alpha}$ with respect to the basis 
\[
\gamma_0^0\otimes \overline{1},\dotsc,\gamma_0^{m-1}\otimes \overline{1},\gamma_0^0\otimes\overline{X},\dotsc,\gamma_0^{m-1}\otimes\overline{X},\dotsc,\gamma_0^0\otimes\overline{X}^{r-1},\dotsc,\gamma_0^{m-1}\otimes\overline{X}^{r-1}.
\]
Then, one can easily see that $A$ satisfies the conditions of Lemma~\ref{lem:tr}(ii). Hence, the assertion follows from the same argument as \cite[Proposition~8]{Tate}.
\end{proof}

\begin{cor}\label{cor:main}
For $V\in\GC$, $\cc{q}{G_K}{V}$ is finite dimensional over $K$ and vanishes for $q>d+1$ and we have
\[
\sum_{q\in\mathbb{N}}{(-1)^q\mathrm{dim}_K\cc{q}{G_K}{V}=0}.
\]
\end{cor}
\begin{proof}
By d\'{e}vissage, this follows from Lemma~\ref{lem:fon}(iv), Theorem~\ref{thm:hyo} and Theorem~\ref{thm:van}.
\end{proof}

\section{Proof of main theorem}\label{sec:Pro}
\begin{lem}\label{lem:ef}
The $\delta$-functor
\[
\CH{\Ll}{\DS{-}}:\GC\to\VEC{\kinf}
\]
is effaceable.
\end{lem}
\begin{proof}
We have an equivalence $(p^{\mathbb{Z}},p^{\mathbb{Z}'}):\GC\cong\GCZ\times\GCZp$ (Corollary~\ref{cor:dec}) and for $V\in\GCZp$, we have $\DS{V}\in\LZp$ by Lemma~\ref{lem:fon}(iii) and $\CH{\Ll}{\DS{V}}=0$ by Lemma~\ref{lem:caL}(i),(ii). So we have only to see the effaceability for $V\in\GCZ$.
By the commutative diagram
\[\xymatrix{
\GCZ\ar[d]^{\DSS}\ar[r]^{((-)^H)^{\mathrm{f}}}_{\cong}&\GZ\ar[r]^{(-)_0}_{\cong}&\GGZ\ar[d]^{\partial}_{\cong}\\
\LZ&&\LLZ\ar[ll]^{\kinf\otimes},
}\]
we have only to check the effaceability of the $\delta$-functor $\CH{\Lk}{-}:\LLZ\to\VEC{K}$. This follows from the equivalence $(p^{\mathbb{Z}},p^{\mathbb{Z}'}):\LK\cong\LLZ\times\LLZp$ (Corollary~\ref{cor:dec}) and Lemma~\ref{lem:KL}, since $\LLl$ is solvable. 
\end{proof}

Let $\iota$ be the morphism of $\delta$-functors from $\GC$ to $\mathrm{Vec}_{\kinf}$
\[
\CH{\Ll}{\DS{-}}\to\kinf\otimes\CC{\gk}{-}
\]
induced by Lemma~\ref{lem:sen}(iv) and Lemma~\ref{lem:ef}.

\begin{lem}\label{lem:cup}
$\iota$ is compatible with the cup product.
\end{lem}
\begin{proof}
By dimension shift, we have only to check at degree~$0$. In this case, the lemma follows from the fact that $\DSS$ is compatible with tensor product (Theorem~\ref{thm:sen} and Lemma~\ref{lem:sen}(i)).
\end{proof}

\begin{lem}\label{lem:degone}
$\iota$ is an isomorphism at degree~$1$.
\end{lem}
\begin{proof}
Let $\mu$ be the morphism of $\delta$-functors from $\GC$ to $\mathrm{Vec}_{\kinf}$
\[
\mu:\CH{\Ll}{\DS{-}}\to\kinf\otimes\mathrm{Ext}_{\GC}(\cp,-)
\]
induced by the effaceability of $\CH{\Ll}{\DS{-}}$. Note that the functorial isomorphism $\mathrm{Ext}_{\GC}^{q}(\cp,-)\to\cc{q}{\gk}{-}$ for $q=0,1$ is compatible with the connecting homomorphisms, that is, for an exact sequence in $\GC$
\[\xymatrix{
0\ar[r]&V'\ar[r]&V\ar[r]&V''\ar[r]&0,
}\]
there is a commutative diagram
\[\xymatrix{
\mathrm{Hom}_{\GC}(\cp,V'')\ar[r]^{\delta}\ar[d]^{\mathrm{can.}}&\mathrm{Ext}^{1}_{\GC}(\cp,V')\ar[d]^{\mathrm{can.}}\\
\cc{0}{\gk}{V''}\ar[r]^{\delta}&\cc{1}{\gk}{V'}.
}\]
where $\delta$ denote the connecting homomorphisms. So, we have only to prove that $\mu$ is isomorphism at degree $1$. Let $\xi:\kinf\otimes\mathrm{Ext}_{\GC}(\cp,-)\to\CH{\Ll}{-}$ be the morphism of $\delta$-functors defined by the composition $\kinf\otimes\mathrm{Ext}_{\GC}(\cp,-)\stackrel{\kinf\otimes\DSS}{\to}\mathrm{Ext}_{\LL}(\kinf,\DS{-})\stackrel{\mathrm{can.}}{\to}\mathrm{Ext}_{U\Ll}(\kinf,\DS{-})=\CH{\Ll}{\DS{-}}$. Then, $\xi\circ\mu$ is identity at degree $0$, which implies $\xi\circ\mu=\mathrm{id}$. Since $\xi$ is injective at degree $1$ (Lemma~\ref{lem:extone}), we obtain the assertion.
\end{proof}

\begin{proof}[Proof of Theorem~\ref{thm:Main}]
We will prove that $\iota$ is an isomorphism. The case degree $\le 1$ is already proved. By d\'{e}vissage and Lemma~\ref{lem:caL}(ii), Theorem~\ref{thm:van}, we can assume $V=\cp(n),n\in\mathbb{Z}$ and in this case, the assertion follows from Theorem~\ref{thm:hyo}(iv), Lemma~\ref{lem:caL}(iv) and Lemma~\ref{lem:cup}.
\end{proof}

\begin{rem}
Define the functor $D:\GC\to\LK$ as the composite
\[\xymatrix{
\GC\ar[r]^{p^{\mathbb{Z}}}&\GCZ\ar[r]^{((-)^H)^{\mathrm{f}}}&\GZ\ar[r]^{(-)_0}&\GGZ\ar[r]^{\partial}&\LLZ\ar[r]^{\mathrm{can.}}&\LK.
}\]
Then, the $\delta$-functor $\CH{\Lk}{D(-)}:\GC\to\VEC{K}$ is effaceable by the same argument as Lemma~\ref{lem:ef} and we have a canonical isomorphism $\cc{0}{\gk}{V}\to\coh{0}{\Lk}{D(V)}$ by the construction. Hence, we have a canonical isomorphism of $\delta$-functors $\CC{\gk}{-}\to\CH{\Lk}{D(-)}$ (note that the $\delta$-functor $\CC{\gk}{-}:\GC\to\VEC{K}$ is effaceable by Theorem~\ref{thm:Main} and Lemma~\ref{lem:ef}).
\end{rem}

\section{Appendix (Errata on \cite{Bri})}\label{sec:App}
Let the notation be as in \cite{Bri}. In this appendix, we point out the following errors in the paper \cite{Bri} and give the argument to fix them.

$\bullet$ p.795, l.16-17. $\dotsb$ $\Gamma'$ s'identifie \`{a} un sous-groupe (distingu\'{e}). On note $\kinf=K_{(\infty)}^{\Gamma'_{\mathrm{tors}}}$, $\dotsc$ 

First, $\Gamma'$ is not distinguished in $G_{K_{(\infty)}/K}$: $\Gamma'$ is not even normal since $K_{(\infty)}^{\Gamma'}=K(t_1^{p^{-\infty}},\dotsc,t_d^{p^{-\infty}})$ is not Galois over $K$. Second, this is more serious, $\Gamma'_{\mathrm{tors}}$ is not normal in $G_{K_{(\infty)/K}}$ in general: $\Gamma'_{\mathrm{tors}}$ is normal if and only if $\Gamma'_{\mathrm{tors}}=1$, that is, $\zeta\in K$ where $\zeta=\zeta_p$ if $p$ is odd, $\zeta=\zeta_{p^2}$ if $p$ is even. So, we cannot put $\Gamma=\mathrm{Gal}(\kinf/K)$ as in l.19 in general.

$\bullet$ p.798, Lemme~1.

In the proof, the author seems to assume that the residue field extension of $K_{m+1}^{(i)}/K_{m}^{(i)}$ is generated by the image of $t_{i}^{p^{-(m+1)}}$, but since the residue field extension of $K(\mu_{p^{m+1}})/K(\mu_{p^m})$ can be inseparable, this does not happen in general. So, the proof of Lemme~1 does not work in general.

To fix these problems, we put the following assumption, similar to \cite[(0-5)]{Hyo}, in $\S 2,\S 3$. Recall that the canonical subfield $\kcan$ of $K$ is the algebraic closure of the fractional of the Witt ring of the maximal perfect subfield of $k$ in $K$.
\begin{ass}\label{ass}
$\zeta\in K$ and the ramification index of $K/\kcan$ is $1$.
\end{ass}
Note that under this assumption, the original proof of Lemme 1 works. Since the original proofs of Th\'{e}or\`{e}me~1, Th\'{e}or\`{e}me~2 in $\S\S 3.1$ are verified only under this assumption, we add the following proofs of general case.
\begin{proof}[Proof of Th\'{e}or\`{e}me~1]
Just put $\mathbb{K}$ as the $p$-adic completion of $K(t_1^{p^{-\infty}},\dotsc,t_h^{p^{-\infty}})$, then the original proof works.
\end{proof}

\begin{dfn}
Let $K$ be a CDVF, $L/K$ be a Galois extension. Let $\mathrm{Rep}_{\ast}{\ast}$ as in notation and convention. For $V\in\mathrm{Rep}_{\widehat{L}}G_{L/K}$, let $V^{\mathrm{f}}$ be the sub-$L$-vector space of $V$ generated by $v\in V$ such that $G_{L/K}\cdot (Kv)$ is finite dimensional over $K$. We call $L/K$ satisfies $\ast$ if
\[
\mathrm{Rep}_{L}G_{L/K}\to\mathrm{Rep}_{\widehat{L}}G_{L/K};V\mapsto\widehat{L}\otimes_{L}V
\]
is an equivalence with a quasi-inverse $V\mapsto V^{\mathrm{f}}$.
\end{dfn}

\begin{lem}
Let $K$ be a CDVF and $L/M/K$ be algebraic extensions with $L/K$ Galois.
\begin{enuroman}
\item If $L/M$ is finite and $M/K$ is Galois, then
\[
L/K\text{ satisfies }\ast\Leftrightarrow M/K\text{ satisfies }\ast.
\]

\item If $M/K$ is finite, then
\[
L/K\text{ satisfies }\ast\Leftrightarrow L/M\text{ satisfies }\ast.
\]
\end{enuroman}
\end{lem}
\begin{proof}
(i) This follows from the equivalences
\[
(-)^{G_{L/M}}:\mathrm{Rep}_{L}{G_{L/K}}\to\mathrm{Rep}_{M}{G_{M/K}},\ \mathrm{Rep}_{\widehat{L}}{G_{L/K}}\to\mathrm{Rep}_{\widehat{M}}{G_{M/K}}
\]
(Hilbert 90) and the commutativity of $(-)^{\mathrm{f}}$ and $(-)^{G_{L/M}}$.

(ii) Since the restriction $G_{L/K}\to G_{L/M}$ commutes with $(-)^{\mathrm{f}}$, the RHS implies the LHS. Assume $L/K$ satisfies $\ast$. Let $\widehat{L}\langle G_{L/K}\rangle,\widehat{L}\langle G_{L/M}\rangle$ be the semilinear group rings of $G_{L/K},G_{L/M}$ over $\widehat{L}$. For $V\in\mathrm{Rep}_{\widehat{L}}G_{L/M}$, let $V':=(\widehat{L}\langle G_{L/K}\rangle\otimes_{\widehat{L}\langle G_{L/M}\rangle} V)|_{G_{L/M}}/V\in\mathrm{Rep}_{\widehat{L}}G_{L/M}$. Then, by the short exact sequence in $\mathrm{Rep}_{\widehat{L}}G_{L/M}$
\[\xymatrix{
0\ar[r]&V\ar[r]&{\widehat{L}\langle G_{L/K}\rangle}\otimes_{\widehat{L}\langle G_{L/M}\rangle}{V}\ar[r]&V'\ar[r]&0,
}\]
we have, by the left-exactness of $(-)^{\mathrm{f}}$ and the snake lemma,
\begin{gather*}
\ker{(\widehat{L}\otimes_{L}(V^{\mathrm{f}})\to V)}=0\\
\mathrm{coker}(\widehat{L}\otimes_{L}(V^{\mathrm{f}})\to V)\subset\ker{(\widehat{L}\otimes_{L}(V')^{\mathrm{f}}\to V')}.
\end{gather*}
This implies that the canonical map $\widehat{L}\otimes_{L}(V^{\mathrm{f}})\to V$ is an isomorphism.

\end{proof}

\begin{proof}[Proof of Th\'{e}or\`{e}me~2]
Follow the notation of \cite{Bri}. Let $K_0$ be a Cohen subfield of $K$ such that $K_0$ contains $t_1,\dotsc,t_h$. (Such a $K_0$ exists by \cite[Theorem~2.1]{Whi}.) Then, we have $(K_0)_{\infty}\subset\kinf$ and $(K_0)_{\infty}/K_0(\zeta)$ satisfies $\ast$. Let $L/K$ be a Galois closure of $\kinf/K_0$. By applying the above lemma to the diagram of algebraic extensions
\[\xymatrix{
(K_0)_{\infty}\ar@{-}[r]\ar@{-}[d]&\kinf\ar@{-}[r]\ar@{-}[dd]&L\\
K_0(\zeta)\ar@{-}[d]&&\\
K_0\ar@{-}[r]&K,
}\]
we obtain the assertion.
\end{proof}

\end{document}